\documentclass[final,11pt]{amsart}

\usepackage{amssymb}
\usepackage{amsthm}
\usepackage{amsmath}
\usepackage{color}
\usepackage{graphicx}
\usepackage{dsfont}

\newcommand{\med}{\mbox{median}}
\newcommand{\R}{{\mathds R}}
\newcommand{\Z}{{\mathds Z}}
\newcommand{\N}{{\mathds N}}

\newcommand{\F}{{\mathcal F}}
\newcommand{\Var}{{\rm Var}}
\newcommand{\Cov}{{\rm Cov}}
\newcommand{\claw}{{\stackrel{\mathcal{D}}{\longrightarrow}}}
\newtheorem{lemma}{Lemma}[section]

\newtheorem{theorem}[lemma]{Theorem}
\newtheorem{proposition}[lemma]{Proposition}

\usepackage{natbib}

\begin{document}
\title{Estimation of the Variance of Partial Sums of Dependent Processes
}
\thanks{The authors were supported in part by the
Collaborative Research Grant 823, Project C3 {\em Analysis of Structural
Change in Dynamic Processes}, of the German Research
Foundation.}

\author[H. Dehling]{Herold Dehling}
\author[R. Fried]{Roland Fried}
\author[O. Sh. Sharipov]{Olimjon Sh.~Sharipov}
\author[D. Vogel]{Daniel Vogel}
\author[M. Wornowizki]{Max Wornowizki}

\address{
Fakult\"at f\"ur Mathematik, Ruhr-Universit\"at Bo\-chum,
44780 Bochum, Germany}
\email{herold.dehling@rub.de}

\address{Fakult\"at Statistik, Technische Universit\"at Dortmund,
44221 Dortmund, Germany}
\email{fried@statistik.tu-dortmund.de}

\address{Department of Probability Theory and Mathematical Statistics,
Institute of Mathematics and Information Technologies, Uzbek Academy of
Sciences, Tashkent, Uzbekistan}
\email{osharipov@yahoo.com}

\address{
Fakult\"at f\"ur Mathematik, Ruhr-Universit\"at Bo\-chum,
44780 Bochum, Germany}
\email{vogeldts@rub.de}

\address{Fakult\"at Statistik, Technische Universit\"at Dortmund,
44221 Dortmund, Germany}
\email{wornowiz@statistik.tu-dortmund.de}

\begin{abstract}
We study subsampling estimators for the limit variance
\[
 \sigma^2=\Var(X_1)+2\, \sum_{k=2}^\infty \Cov(X_1,X_k)
\]
of partial sums of a stationary stochastic process $(X_k)_{k\geq 1}$. We establish $L_2$-consistency of a non-overlapping block resampling method. Our results apply to processes that can be represented as functionals of strongly mixing processes. Motivated by recent applications to rank tests, we also study
estimators for the series $\Var(F(X_1))+2\, \sum_{k=2}^\infty \Cov(F(X_1),F(X_k))$, where $F$ is the
distribution function of $X_1$. Simulations illustrate the usefulness of the proposed estimators and of a mean squared error optimal rule for the choice of the block length.
\end{abstract}


\maketitle


\section{Introduction and Main Results}
Let $(X_i)_{i\geq 1}$ be a stationary ergodic sequence of random variables
with mean $\mu$ and finite variance, satisfying the central limit theorem, i.e.
\begin{equation*}
 \frac{1}{\sqrt{n}} \sum_{i=1}^n (X_i-\mu) \claw N(0,\sigma^2),
\end{equation*}
where the limit variance $\sigma^2$ is given by the formula
\begin{equation}
 \sigma^2=\Var(X_1)+2\sum_{k=2}^\infty \Cov(X_1,X_k).
\label{eq:clt-var}
\end{equation}
For practical applications of the central limit theorem,
e.g., for calculation of confidence intervals or testing hypothesis concerning the mean,
we need to estimate the variance $\sigma^2$.

We investigate the properties of a subsampling estimator of $\sigma$.
 Given observations $X_1,\ldots,X_n$ and a block length $l=l_n$, we define the overall mean and the block means as
\[ 
	\bar{X}_n =\frac{1}{n} \sum_{j=1}^n X_j,  \quad \mbox{ and } \quad
	\frac{1}{l}S_i(l) = \frac{1}{l}\sum_{j=(i-1)\, l+1}^{i\, l} X_j,
\] 
respectively. By the central limit theorem, $(S_i(l)-l\, \mu)/\sqrt{l}$ converges in distribution to
$N(0,\sigma^2)$ as $l$ goes to infinity, and thus $E|S_i(l)-l\, \mu|/\sqrt{l} \rightarrow \sigma \sqrt{2/\pi}$. Thus, it
is natural to estimate $\sigma$ by the arithmetic mean of $\sqrt{\frac{\pi}{2}} \frac{|S_i(l)-l\, \mu|}{\sqrt{l}}$, $i=1,\ldots,[n/l]$. Replacing the unknown mean $\mu$ by the sample mean $\bar{X}_n$, we
obtain
\begin{equation}
	\hat{B}_n=\frac{1}{[n/l]} \sqrt{\frac{\pi}{2}} \sum_{i=1}^{[n/l]} \frac{|S_i(l)-l\bar{X}_n|}{\sqrt{l}}
\label{eq:B_n}
\end{equation}
as an estimator for $\sigma$.
In what follows we will always write $k_n=[n/l_n]$ for the number of full blocks of length $l_n$ in the set $\{1,\ldots,n\}$.

Subsampling  estimators of the limit variance $\sigma^2$ of the arithmetic mean of  a stationary stochastic process have been studied by various authors under different assumptions concerning the dependence structure.
\citet{carlstein:1986} investigated $\alpha$-mixing processes and the estimator
\begin{equation}
\label{eq:carlstein}
\hat{B}_n^{(2)}=\frac{1}{[n/l]}  \sum_{i=1}^{[n/l]} \left(\frac{|S_i(l)-l\bar{X}_n|}{\sqrt{l}}\right)^2
\end{equation}
for $\sigma^2$.  \citet{peligrad:shao:1995} studied properly normalized means of $p$-th powers of
$|S_i(l)-l\bar{X}_n|/\sqrt{l}$ as an estimator for $\sigma^p$ in the case of $\rho$-mixing processes. They consider overlapping blocks, i.e.\ 
\[
 \tilde{B}_n^{(p)}= \frac{c_p}{n-l+1}  \sum_{i=0}^{n-l} \left(\frac{|\tilde{S}_i(l)-l\bar{X}_n|}{\sqrt{l}}\right)^p,
\]
where $\tilde{S}_i(l)=\sum_{j=i+1}^{i+ l} X_j$. \citet{peligrad:suresh:1995} studied the same estimator, in the
case $p=1$, for associated processes. \citet{doukhan:2010} investigated the case of
weakly dependent \citep[in the sense of][]{doukhan:louhichi:1999} processes.

In this paper, we extend the above mentioned results to processes that can be represented as
functionals of  strongly mixing processes. We assume that $(X_i)_{i\geq 1}$ can be
written as $X_i=f((Y_{i+k})_{k\in \Z})$ for all $i\geq 1$, where $f:\R^\Z\rightarrow \R$ is a measurable function and where $(Y_i)_{i\geq 1}$ is an $\alpha$-mixing process. Recall that a process $(Y_i)_{i\geq 1}$ is called $\alpha$-mixing (or strongly mixing) if
\begin{equation*}
\alpha_k\ = \ \sup_{n\geq 1}\, \sup_{A\in \F_{\!\!-\infty}^n, B\in \F_{n+k}^\infty} |P(A\cap B)-P(A)\, P(B)|
\rightarrow 0
\end{equation*}
as $n\rightarrow \infty$. Here $\F_k^m$ denotes the $\sigma$-field generated by the random variables
$Y_{k},Y_{k+1},\ldots,Y_m$. The coefficients $\alpha_k$ are called mixing coefficients. In addition, we have to require that the function $f:\R^\Z\rightarrow \R$ defined above is continuous in a suitable sense. Define
\begin{eqnarray*}
  \psi_m&=&
 E\left| f((Y_{i+k})_{k\in \Z}) -E(f((Y_{i+k})_{k\in \Z})|
 Y_{i-m},\ldots,Y_{i+m}) \right|^{\frac{2+\delta}{1+\delta}},
\\
 \phi_m&=&E\left| f((Y_{i+k})_{k\in \Z}) -E(f((Y_{i+k})_{k\in \Z})|
 Y_{i-m},\ldots,Y_{i+m}) \right|.
\end{eqnarray*}
In what follows, we will have to make assumptions concerning the behavior of the coefficients
$\alpha_m$, $\psi_m$ and $\phi_m$ as $m\rightarrow \infty$.

Functionals of $\alpha$-mixing processes cover most of the standard examples of weakly dependent processes known in the literature, for example ARMA- and GARCH-processes from time series analysis, many Markov processes and hidden Markov models, the sequence of digits and remainders in continued fraction expansion, and many chaotic dynamical systems, such as expanding, piecewise monotone maps of the unit interval. Details can be found, e.g.,
in \citet{borovkova:burton:dehling:2001} and in \citet{bradley:2007}.
\begin{theorem}
Suppose $(X_k)_{k\geq 1}$ is a stationary process that can be expressed as a functional of an
$\alpha$-mixing process. Suppose that $E|X_1|^{2+\delta}<\infty$ for some $\delta>0$ and that the approximation coefficients $(\psi_k)_{k\geq 1}$ satisfy
\[
	\sum_{k=1}^\infty (\psi_k)^{\frac{1+\delta}{2+\delta}} <  \infty
\] 
and the mixing coefficients $(\alpha_k)_{k\geq 1}$ satisfy $\alpha_k = O(k^{-\mu})$ for some $\mu > (2+\delta)^2/\delta$.
Then, for any sequence $(l_n)_{n\geq 1}$ satisfying $l_n \nearrow \infty$ and $l_n=o(n)$, we have $\hat{B}_n \to  \sigma$ in $L_2$.
\label{th:fctl-ltwo}
\end{theorem}

%
%
%
%
%
%
%
\citet{dehling:fried:2011} study robust non-parametric tests for structural breaks in time series. They consider the model $X_i=\xi_i+\mu_i$,
where $(\xi_i)$ is a stationary stochastic process and $(\mu_i)_{i\geq 1}$ is a sequence of unknown constants. The hypothesis of interest is
\[
 H_0:\quad \mu_1=\ldots=\mu_n,
\]
to be tested against the alternative $\mu_1=\ldots=\mu_{n_1}\neq \mu_{n_1+1}=\ldots=\mu_n$.
 Investigating the asymptotic distribution of the two-sample Hodges-Lehmann test statistic
\[
 Q_{n_1,n_2}(0.5)=\med\{ (X_j-X_i) : 1\leq i \leq n_1, n_1+1\leq j \leq n_1+n_2 \}.
\]
Dehling and Fried show that under the above null hypothesis, we have
\[
   H^\prime(0) \sqrt{\frac{n_1\, n_2}{n_1+n_2}} \frac{Q_{n_1,n_2}(1/2) }{\sqrt{\sigma_F^2}} \ \claw \ N(0,1)
\]
as $n_1,n_2\rightarrow \infty$.
Here $H(x)=P(X_1- \tilde{X}_1 \leq x)$, where $\tilde{X}_1$ is an independent copy of $X_1$, and
\begin{equation}
 \sigma_F^2=\Var(F(X_1))+2\, \sum_{k=2}^\infty \Cov(F(X_1),F(X_k)),
\label{eq:clt-var2}
\end{equation}
where $F$ is the distribution function of $X_1$. For statistical applications, the limit variance $\sigma_F^2$
has to be estimated.
We  can directly apply Theorem~\ref{th:fctl-ltwo} to the sequence $(F(X_k))_{k\geq 1}$, provided that it satisfies the conditions and that $F$ is known. The latter is generally not the case and thus $F$ has to be replaced by the empirical distribution function $\hat{F}_n$. In what follows, we will show that this leads to an $L_2$-consistent estimator.

Denoting the empirical distribution function by $\hat{F}_n(x)=\frac{1}{n}\sum_{i=1}^n 1_{\{X_i\leq x\}}$, we define the estimators
\begin{eqnarray}
	\nonumber
 D_n&=&\frac{1}{[n/l_n]} \sqrt{\frac{\pi}{2}} \sum_{i=1}^{[n/l_n]} \frac{|T_i(l_n)-l_n\, \bar{U}_n|}{\sqrt{l_n}}, \\
 \hat{D}_n&=& \frac{1}{[n/l_n]} \sqrt{\frac{\pi}{2}}
  \sum_{i=1}^{[n/l_n]} \frac{|\hat{T}_i(l_n)-l_n\, \tilde{U}_n|}{\sqrt{l_n}},
\label{eq:def-dhat}
\end{eqnarray}
where $\bar{U}_n=\frac{1}{n}\sum_{j=1}^n F(X_j)$, $\tilde{U}_n=\frac{1}{n}\sum_{j=1}^n \hat{F}_n(X_j)$, and
\[
 T_i(l) =  \sum_{j=(i-1)l+1}^{il} F(X_j), \qquad 
 \hat{T}_i(l) =  \sum_{j=(i-1)l+1}^{il} \hat{F}_n(X_j).
\]
\begin{theorem}
Let $(X_k)_{k\geq 1}$ be a stationary process satisfying the assumptions of Theorem~\ref{th:fctl-ltwo}. In addition, assume that $\alpha_n=O(n^{-8})$, $\phi_m=O(m^{-12})$ and that $F(x)=P(X\leq x)$ is Lipschitz-continuous. Then, as $n\to \infty$, $l_n \to \infty$ and $l_n=o(\sqrt{n})$, we have $\hat{D}_n \to \sigma_F$ in $L_2$.
\label{th:edf-ltwo}
\end{theorem}
%
%
%
%
%
%
%
%
%
%
%
%
%
%
%
\section{Simulation study}
We investigate the performance of the proposed estimators and compare them to alternatives. The estimators are applied to ARMA(1,1) pro\-cess\-es  with Gaussian innovations and different parameter settings, which are given in Table~\ref{ts_coefs_tab}. For each parameter combination, 1000 samples are drawn, and variance, bias and mean squared error of the estimates are computed. The sample size is $n=500$ throughout.
\begin{table}[htb]
\centering
\begin{tabular}[c]{|r|c|c|c|c|c|c|c|c|c|c|c|}
\hline
AR &  0.5 	& 0.1	& -0.1& -0.8&0 	& 0 	& 0 	& 0 	& 0.5& -0.5& 0 \\ 
\hline
MA &  0 	&	0 	&	0 	&0 		&0.8& 0.1	& -0.8& -0.1& 0.5& -0.5& 0	\\
\hline
\end{tabular}
\caption{Types of ARMA(1,1) time series used in the simulation. AR and MA denote the corresponding coefficents of the ARMA(1,1) model in the usual notation.}
\label{ts_coefs_tab}
\end{table}
Besides comparing several estimators, another question addressed by the simulations is the applicability of Carlstein's rule for choosing the block length. \citet{carlstein:1986} proved it to be optimal in terms of mean squared error for his estimator $\hat{B}_n^{(2)}$, cf.~(\ref{eq:carlstein}), in case of an AR(1) process with known autocorrelation parameter.

The simulation study consists of two stages: the first part is concerned with the estimation of $\sigma^2$, cf.~(\ref{eq:clt-var}), the second part with the estimation of $\sigma_F^2$, cf.~(\ref{eq:clt-var2}).
\paragraph{Part 1} We compare $\hat{B}_n$, cf.~(\ref{eq:B_n}), the Carlstein estimator $\hat{B}_n^{(2)}$, referred to as CE in the following, and the kernel-based estimator of \citet{dejong:davidson:2000}, abbreviated by JDE in the following. While the former two depend on the block length, the latter is affected by the choice of the kernel and the bandwidth $\gamma_n$. We use the Bartlett kernel, which is a common choice due to guaranteed positive semi-definiteness.

To compute mean squared error and bias we compare the estimates to the true value of 
	$\Var(\frac{1}{\sqrt{n}} \sum_{i=1}^n X_i) = \Var(X_1) + \frac{2}{n} \sum_{k = 2} ^{n} \Cov(X_1, X_k) (n+1-k)$,
not the infinite series (\ref{eq:clt-var}). We use the fixed block lengths 1, 2, 5, 10, 30, 50 and 100, and determine the data adaptive block length following Carlstein as follows: An AR(1) process is fit to the sample (regardless of whether it was generated as such) using the R function {\tt arima()} \citep{r-language}, and the block length is computed from the estimated autocorrelation parameter. For the JDE, the bandwidths  2.8, 3.5, 4.7, 7.9 and 22.4 are used, corresponding to sixth to second root of $n = 500$.

\paragraph{Part 2} The second part of the simulation, which deals with the estimation of $\sigma_F^2$, compares $\hat{D}_n$, cf.\ (\ref{eq:def-dhat}), to the \citet{peligrad:shao:1995} estimator, denoted by PSE in the following. Both depend on a block length, but contrary to $\hat{D}_n$, the PSE is based on overlapping blocks. The data generation is exactly the same as in the first part (sample size, number of repetitions and ARMA parameters), but the data evaluation is different: Each sample is split into halves and the estimate of interest is computed as the mean of the two subsample estimates. This choice is motivated by the intended application to the change-point test by \citet{dehling:fried:2011}. 
The block lengths considered are 2, 5, 10, 25, 40, 50  for $\hat{D}_n$ and 5, 10, 25, 40, 45 for the PSE. Note that, while a block length 1 is of interest in the first simulation, since it yields a variance estimator for independent observation, it leads to data independent estimators in this situation and is hence neglected.

For computing the true value of $\sigma^2_F$ we make use of the fact that the population version of Spearman's rho $\varrho_s(X,Y) = 12 E\left(F(X)F(Y)\right) - 3$ 
equals $\frac{6}{\pi} \arcsin \left(\frac{\rho}{2}\right)$ if $X$ and $Y$ are jointly Gaussian with correlation $\rho$. This result can be traced back to \citet{pearson:1907}; for a newer reference see, e.g., \citet{croux:dehon:2010}. Hence, letting $\rho_k$ denote the correlation of $X_1$ and $X_k$, we have
\[
	\Cov\left(F(X_1),F(X_k)\right) 
	\ = \ \frac{1}{12} \varrho_s \ = \ \frac{1}{2 \pi} \arcsin\left(\frac{\rho_k}{2}\right).
\] 
\begin{figure}[t]
\begin{minipage}{0.49\linewidth}
\includegraphics[width=\textwidth]{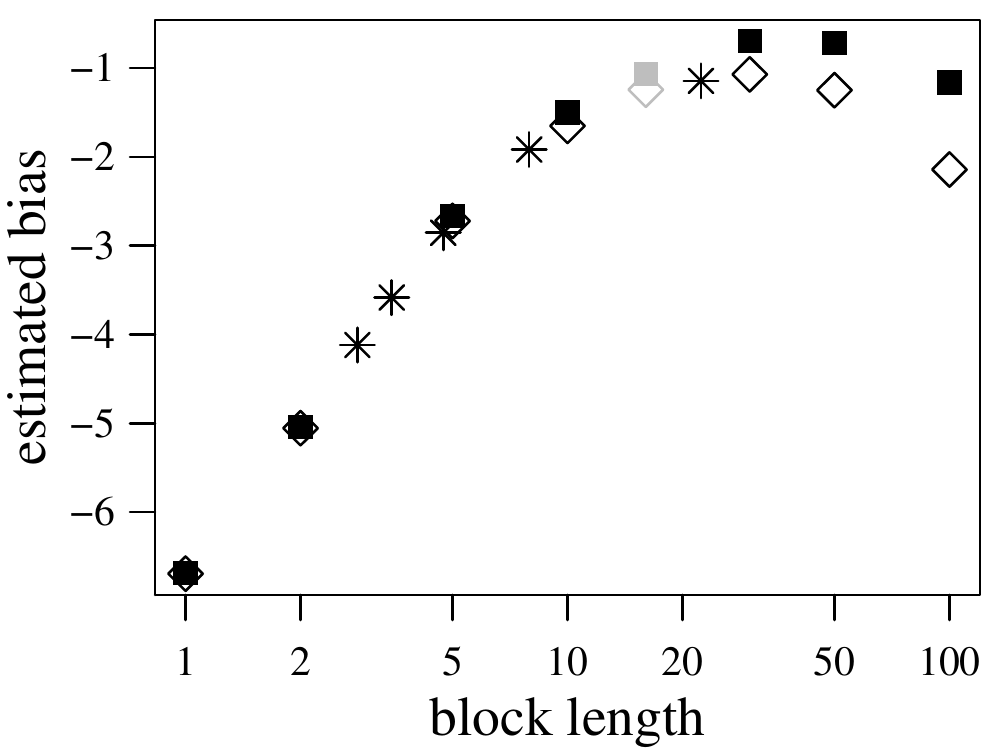}
\end{minipage}
\hfill 
\begin{minipage}{0.49\linewidth}
\includegraphics[width=\textwidth]{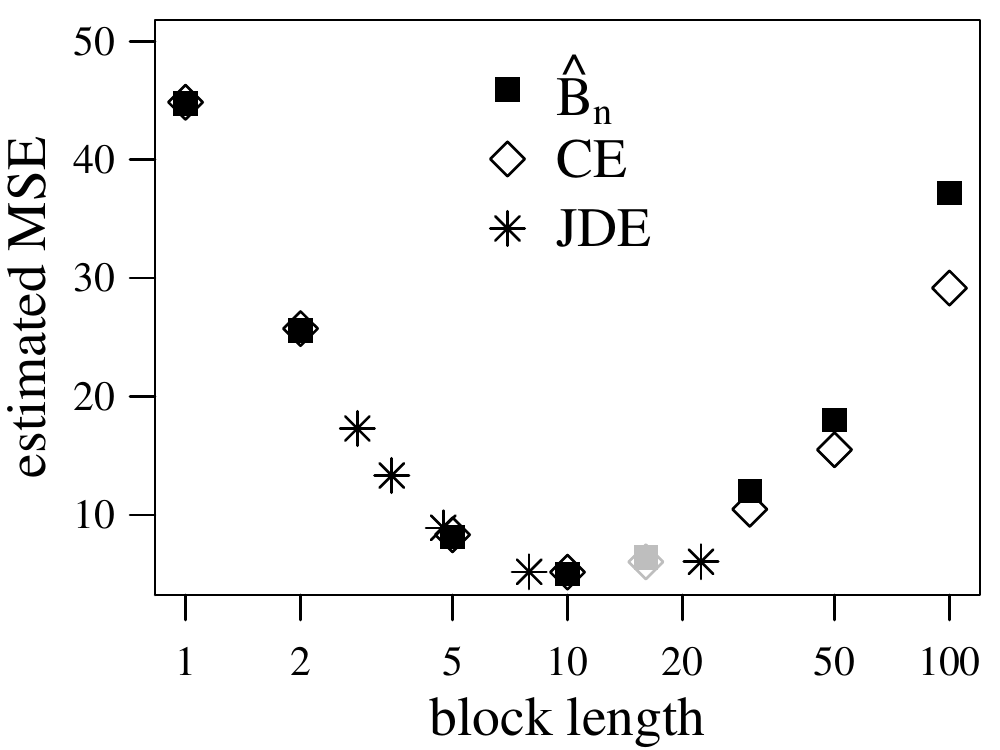}
\end{minipage}
\caption{Estimated bias (left) and mean squared error (right) of the estimators $\hat{B}_n$, CE and JDE as a function of the smoothing parameter; computed from 1000 samples of size 500 generated from an ARMA(0.5, 0.5) model. 
Gray symbols correspond to the data adaptive block length, plotted at the mean block length of 16.04. The x-axis is on log scale.
}
\label{fig:Bias.MSE.1}
\end{figure}

\paragraph{Results}  
As far as the comparison of the estimators is concerned, we get similar results for all ARMA parameter settings, which are exemplified for the $\sigma^2$-estimators at an ARMA(0.5,0.5) process in Figure \ref{fig:Bias.MSE.1} and for the $\sigma^2_F$-estimators at an MA(0.1) process in Figure \ref{fig:Bias.MSE.2}. The performance of the estimators is generally quite similar. This is also true for $\hat{D}_n$ and the PSE, which may seem surprising, since the PSE, in contrast to $\hat{D}_n$, uses overlapping blocks.  In terms of bias, $\hat{B}_n$ and $\hat{D}_n$  slightly outperform their competitors (see left panels of Figures \ref{fig:Bias.MSE.1} and \ref{fig:Bias.MSE.2}), but they show slightly bigger variances, leading to a worse mean squared error (see right panels).  
The choice of the smoothing parameter (block length or bandwidth) has altogether a much larger influence than the choice of the method.
We furthermore observe the following:
The data adaptive block lengths, which correspond to the gray symbols in the figures, perform quite well in terms of mean squared error (MSE) unless the MA-coefficient is strongly negative.
Except for a few cases, $\sigma^2$ is always underestimated if all autocorrelations of the underlying process are non-negative and overestimated if all are non-positive.
In the white noise case, a block length of 1 is MSE optimal, but other block lengths give very similar results. On the other hand, even for small non-zero AR- or MA-coefficients, not accounting for the dependence substantially impairs the accuracy of the estimation.

In summary, all estimators considered show a very similar performance for good choices of the smoothing parameter -- despite some considerable differences in their construction. An unfavorable choice of the block length leads to an increase in both, bias and variance, but with varying extents for the different estimators. Based on our results, we suggest to use the estimators $\hat{B}_n$, $\hat{D}_n$  if one is willing to trade a smaller bias for a bigger variance. 
The data adaptive block length proposed by \citet{carlstein:1986} seems to yield acceptable results for all block length dependent estimators even if the assumption of an AR(1) process is not met. 
We also carried out some simulations for other values of $n$. While the estimators' performance, for fixed block length respectively bandwidth, clearly improves with the sample size, the ranking of the estimators is little affected. Therefore we focus on a fixed sample size here.

\begin{figure}[t]
\begin{minipage}{0.49\linewidth}
\includegraphics[width=\textwidth]{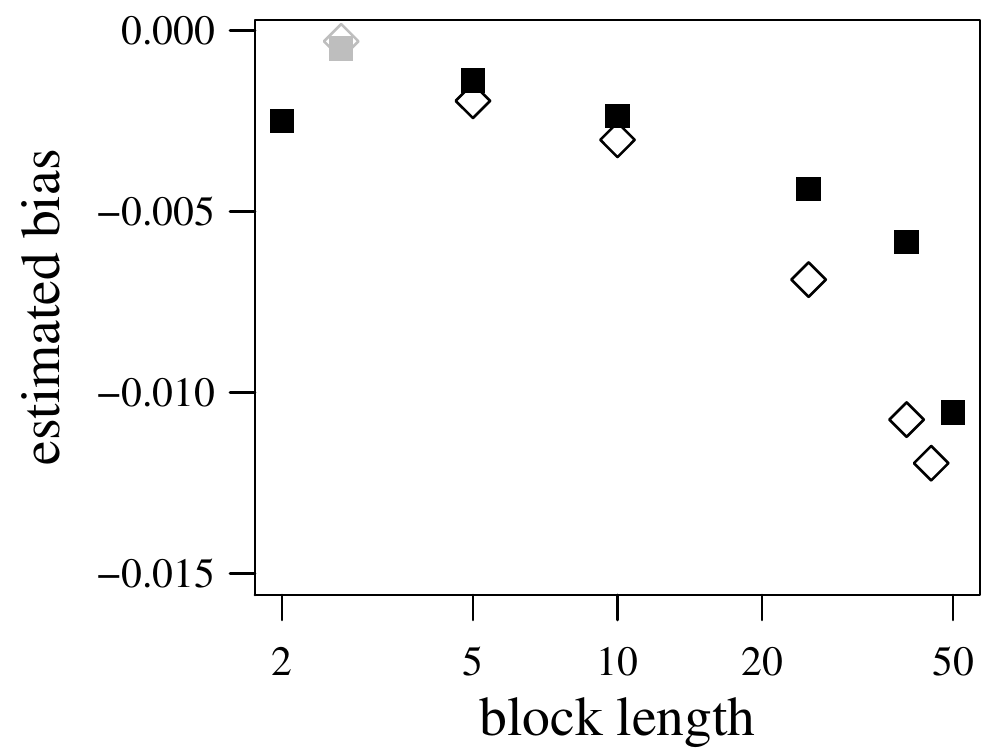}
\end{minipage}
\hfill 
\begin{minipage}{0.49\linewidth}
\includegraphics[width=\textwidth]{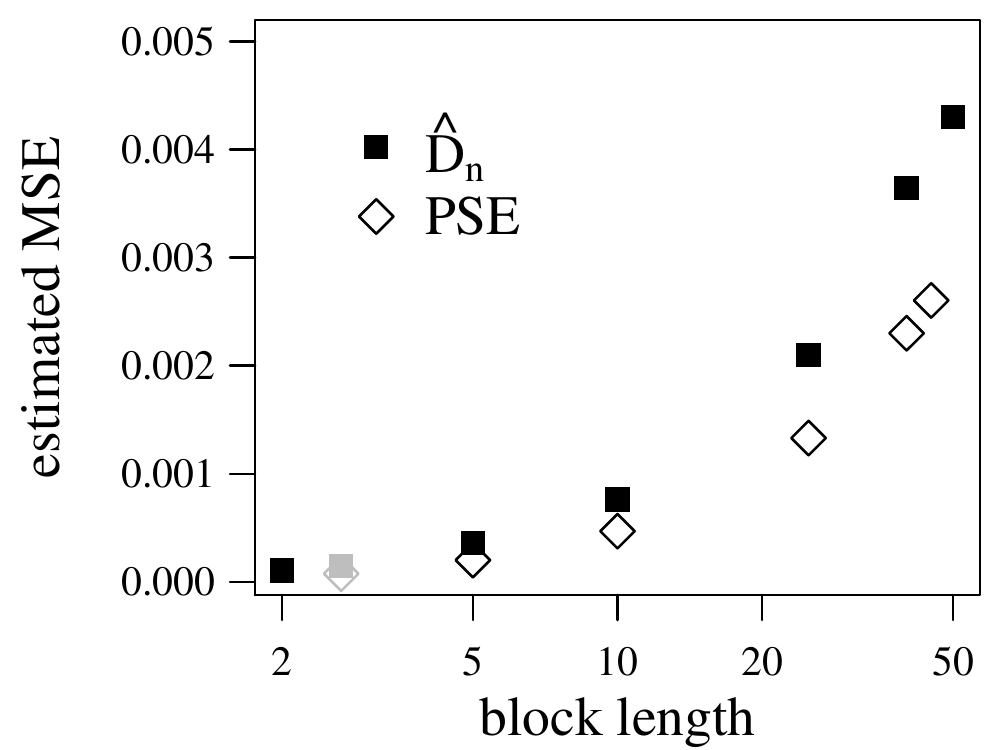}
\end{minipage}
\caption{Estimated bias (left) and mean squared error (right) of the estimators $\hat{D}_n$ and PSE as a function of the block length; computed from 1000 samples of size 500 generated from an ARMA(0.5, 0.5) model.
Gray symbols correspond to the data adaptive block length, plotted at the mean block length of 2.7. The x-axis is on log scale.
}
\label{fig:Bias.MSE.2}
\end{figure}
%
%
%

%
%
%
%
%
%
%
%
%
%
%
%
%
%
%
%
%
%
%

%
%
%
%
%
%
%
%

\section{Proofs}
We will first state and prove a special case of Theorem~\ref{th:fctl-ltwo} where the process $(X_i)_{i\geq 1}$
is $\alpha$-mixing. We need this result later on in the proof of Theorem~\ref{th:fctl-ltwo}.

\begin{proposition}
Suppose $(X_k)_{k\geq 1}$ is a stationary, $\alpha$-mixing process with\\ $E(|X_1|^{2+\delta}) < \infty$ for some $\delta >0$ and mixing coefficients $(\alpha_k)_{k\geq 1}$ satisfying 
$\alpha_k = O(k^{-\mu})$ for some $\mu > (2+\delta)^2/\delta$. Then, for any sequence $(l_n)_{n\geq 1}$ satisfying $l_n \nearrow \infty$ and $l_n=o(n)$, we have $\hat{B}_n \to \sigma$ in $L_2$.
\label{prop:sm-ltwo}
\end{proposition}
\begin{proof} We will show that $ E\left|\hat{B}_n - \sigma \right|^2\rightarrow 0$
as $n\rightarrow \infty$. By definition of $\hat{B}_n$, we get
\[
 \left|\hat{B}_n-  \sigma  \right|
 \leq \left|\frac{1}{k} \sqrt{\frac{\pi}{2}} \sum_{i=1}^k \frac{|S_i(l_n)-l_n\, \mu|}{\sqrt{l_n}}
 - \sigma \right|  +   \sqrt{\frac{\pi}{2}} \sqrt{l_n}(\bar{X}_n-\mu) \quad.
\]
As $l_n=o(n)$, we get by the central limit theorem
$\sqrt{l_n}(\bar{X}_n-\mu) \rightarrow 0$ in $L^2$. Thus, it remains to prove
that
\[
 E\left|\frac{1}{k}\sum_{i=1}^k \sqrt{\frac{\pi}{2}} \frac{|S_i(l_n)-l_n \, \mu|}{\sqrt{l_n}}
  - \sigma \right|^2\rightarrow 0.
\]
By Lemma \ref{lem3} we have
\[
 E\left(\frac{|S_i(l_n)-l_n\, \mu|}{\sqrt{l_n}} \right)
 \rightarrow \sqrt{2/\pi}\,  \sigma
\]
and thus it suffices to show
\[
 E\left|\frac{1}{k}\sum_{i=1}^k \frac{|S_i(l_n)-l_n \, \mu|}{\sqrt{l_n}}
  -E\left( \frac{|S_i(l_n)-l_n \, \mu|}{\sqrt{l_n}} \right)     \right|^2
\rightarrow 0.
\]
Applying Lemma \ref{lem1} with $t=0$ to the stationary $\alpha$-mixing process $(S_i(l))_{i\geq 1}$, we obtain
\begin{eqnarray*}
 	& & E \left| \frac{1}{k} \sum_{i=1}^{k}\dfrac{\lvert S_{i}(l_n)-l_n\mu\rvert-
E\lvert S_{1}(l_n)-l_n\mu\rvert}{\sqrt{l_n}}\right|^{2}\\
 \qquad  &=&
 \frac{1}{k^{2}l_n}E\left| \sum_{i=1}^{k}\left(\lvert S_{i}(l_n)-l_n\mu\rvert-
E\lvert S_{1}(l_n)-l_n\mu\rvert\right)\right|^{2}\\
&\leq& \frac{C}{k \, l_n}
 \left( E \big|\left( \lvert S_{1}(l_n)-l_n\mu\rvert-
 E\lvert S_{1}(l_n)-l_n\mu\rvert \right) \big|^{2+\frac{\delta}{2}}\right)^{\frac{2}{2+\delta/2}}  \\
&\leq & \frac{C}{k \, l_n} \left(E\lvert S_{1}(l_n)-l_n\mu\rvert^{2+\frac{\delta}{2}}\right)^
{\frac{2}{2+\delta/2}} 
\ \ \leq \ \ \frac{C}{k}   \left(E\lvert X_1\rvert^
{2+\delta}\right)^{\frac{2}{2+\delta}},
\end{eqnarray*}
where we have used Lemma \ref{lem1} with $t=\delta/2$ in the final step. As $k=k_n\rightarrow \infty$, we have
thus proved Proposition~\ref{prop:sm-ltwo}.  %
\end{proof}

\begin{proof}[Proof of Theorem~\ref{th:fctl-ltwo}]
Defining
\[ 
 \xi_i(s) = E(X_i|\F_{i-s}^{i+s}), \qquad
\eta_i(s) = X_i -  E(X_i|\F_{i-s}^{i+s}),
\]
we decompose $(X_i)_{i\geq 1}$ into a part that is $\alpha$-mixing
and a remaining part, which is  small. In addition, we  define the corresponding arithmetic means
$\overline{\xi}_n=\frac{1}{n} \sum_{i=1}^n \xi_i(s)$ and
$\overline{\eta}_n=\frac{1}{n} \sum_{i=1}^n \eta_i(s)$
and the block sums
\[
 S_{i}(\xi,l) = \! \sum_{j=(i-1)l+1}^{il}\!  \xi_j(s), \qquad
 S_{i}(\eta,l)= \! \sum_{j=(i-1)l+1}^{il}\! \eta_j(s), \qquad i \in \N.
\]
Furthermore, let
\[
	\sigma^2(s)=\Var(\xi_1(s))+2\sum_{i=2}^{\infty} \Cov(\xi_{1}(s),\xi_{i}(s)).
\]
For fixed $s$, the process $\xi_i(s)$ is $\alpha$-mixing with coefficients
$\alpha_{i-2s}$. Moreover, we have $E(|\xi_i(s)|^{2+\delta})<\infty$ and
$E(|\eta_i(s)|^{2+\delta})<\infty$. Finally, one can show that
\begin{equation*}
 |\Cov(\eta_0(s),\eta_j(s))  | =O\left(\alpha_{[j/3]}^{\delta/(2+\delta)}
 +\psi_{[j/3]}^{(1+\delta)/(2+\delta)}\right),
\end{equation*}
see \citet{ibragimov:linnik:1971}.  Observe that
\begin{eqnarray}
 \left|\hat{B}_n - \sigma  \right|
 &\leq&  \left|\frac{1}{k} \sqrt{\frac{\pi}{2}}  \sum_{i=1}^k
 \frac{|S_i(l,\xi)-l\, \overline{\xi}_n|}{\sqrt{l}} -
   \sigma(s)   \right|  + |\sigma-\sigma(s)|
\label{eq:main-th}    \\
&& \qquad +\frac{1}{k} \sqrt{\frac{\pi}{2}}\sum_{i=1}^k
 \frac{|S_i(l,\eta)-l\, \overline{\eta}_n|}{\sqrt{l}}\nonumber \\
&\leq &  \left|\frac{1}{k} \sqrt{\frac{\pi}{2}}  \sum_{i=1}^k
 \frac{|S_i(l,\xi)-l\, \overline{\xi}_n|}{\sqrt{l}} -
   \sigma(s)   \right|  + |\sigma-\sigma(s)| \nonumber \\
&& \qquad  + \frac{1}{k} \sqrt{\frac{\pi}{2}} \sum_{i=1}^k
 \frac{|S_i(l,\eta)-l\, E(\eta_1(s))|}{\sqrt{l}}
 +\sqrt{l}|E(\eta_1(s)) - \overline{\eta}_n|.
\nonumber
\end{eqnarray}
We will now bound each of the terms on the right hand side separately.
To the third term, we apply the inequality
\[ 
 E\left(\frac{S_i(l,\eta) -l E\eta_1(s)}{\sqrt{l}} \right)^2
  \leq  \Var(\eta_1(s))+2\sum_{j=2}^l\left(1-\frac{j}{l}\right) \Cov(\eta_1(s),\eta_j(s)) 
\]\[
  \leq \, (2N+1)\Var(\eta_1(s)) + 2 \, \sum_{j=N}^l\left((\alpha_{j/3}^{\delta/(2+\delta)}
  + \psi_{j/3}^{(1+\delta)/(2+\delta)}   \right),
\]
valid for any $2\leq N\leq l$. We fix $N$ and choose $s_0$ sufficiently large such that for all $s\geq s_0$
\[
  E\left(\frac{S_i(l,\eta) -l E\eta_1(s)}{\sqrt{l}}\right)^2 \leq \epsilon^2
\]
and hence by Minkowski's inequality
\[
 \left\| \frac{1}{k}\sum_{i=1}^k
  \frac{|S_i(l,\eta)-l\, E(\eta_1(s))|}{\sqrt{l}} \right \|_2\leq \epsilon.
\]
Since $\sigma(s)\rightarrow \sigma$, we also have $|\sigma(s)-\sigma|\leq \epsilon$ for all
$s\geq s_0$ with $s_0$ sufficiently large. 
Proposition~\ref{prop:sm-ltwo} implies that, for any fixed $s$, 
\[
 \frac{1}{k}\sqrt{\frac{\pi}{2}}\sum_{i=1}^k
 \frac{|S_i(l,\xi)-l\, \overline{\xi}_n|}{\sqrt{l}}
 \rightarrow \sigma(s).
\]
in $L_2$ as $n\rightarrow\infty$.
Hence the first term on the right hand side of (\ref{eq:main-th}) converges to $0$ in
$L_2$. Finally, stationarity and Lemma \ref{lem2} imply  $\sqrt{l} \left|\overline{\eta}_n- E(\eta_1(s)) \right|  \rightarrow 0$ in $L_2$ as $n\rightarrow\infty$. This finishes the proof of Theorem~\ref{th:fctl-ltwo}. 
\end{proof}
\begin{proof}[Proof of Theorem~\ref{th:edf-ltwo}]
As $F$ is Lipschitz continuous, the process $(F(X_k))_{k\geq 1}$ satisfies the conditions of
Theorem~\ref{th:fctl-ltwo}, hence $D_n\rightarrow \sigma_F$ in $L_2$ by Theorem~\ref{th:fctl-ltwo}. Thus it suffices to show that $\hat{D}_n-D_n \rightarrow 0$ almost surely. We get
\begin{eqnarray*}
	|D_n-\hat{D}_n|
	\leq \, && \hspace{-1em} \frac{1}{[n/l_n]\sqrt{l_n}} \sqrt{\frac{\pi}{2}} \sum_{i=1}^{[n/l_n]}
 	\left|  T_i(l_n)-\hat{T}_i(l_n)-l_n(\bar{U}_n-\tilde{U}_n )\right| \\
%
 \leq 
	\frac{\sqrt{\pi/2}}{[n/l_n]\sqrt{l_n}}  \sum_{i=1}^{[n/l_n]} && \hspace{-2em}
 	\left(\sum_{j=(i-1)l+1}^{il}|\hat{F}_n(X_j)-F(X_j)| + \frac{l_n}{n}\sum_{i=1}^n|\hat{F}_n(X_i)-F(X_i)|    \right) \\[0.5ex]
 \leq \, &&  \hspace{-1em} 2 \sqrt{l_n}\sup_{t\in \R} |\hat{F}_n(t)-F(t)|. 
\end{eqnarray*}
Lemma \ref{lem:ep-lil} implies $|D_n-\hat{D}_n|\rightarrow 0$ almost surely, provided $l_n=o(\frac{n}{\log\log n})$. This proves that $\hat{D}_n$ is an $L_2$ consistent estimator of the limit variance $\sigma_F$.
\end{proof} 
\section{Auxiliary results}
For ease of reference we state some results from the literature on weakly dependent processes that
we need in the course of the proofs.
\begin{lemma}\label{lem1}{\rm \citep{yokoyama:1980}}
Let $(X_i)_{i\geq 1}$ be a stationary sequence
of random variables with $EX_1=\mu$,
\[
	E\lvert X_1\rvert^{2+\delta}<\infty \quad \mbox{and}\quad
 \sum_{k=1}^{\infty}n^{\frac{t}{2}-1}
 \left( \alpha_k\right) ^{\frac{2+\delta-t}{2+\delta}}<\infty,
\]
for some $0<\delta \leq \infty$  and $2 \leq t < 2+\delta$. Then
\[
E\lvert\sum_{k=1}^{n}(X_i-\mu)\rvert^{t}\leq Cn^{\frac{t}{2}}\left(
E\lvert X_1\rvert^{2+\delta}\right) ^{\frac{t}{2+\delta}}.
\]
\end{lemma}
\begin{lemma}\label{lem2}{\rm \citep{shao:yu:1993}}
Let $(\xi_i)_{i\geq 1}$ be a sequence of random variables with $E\xi_n=0$ and $\sup E\xi_n^{2}<\infty$. Assume that there is a constant $C>0$ such that $\sup E\left(\sum_{i=k+1}^{k+n}\xi_i\right)^{2}\leq C n$ for any $n\geq 1$. Then
\begin{eqnarray*}
\limsup\dfrac{\lvert\sum_{i=1}^{n}\xi_i\rvert}{\sqrt{n}\log^{2}n}=0 \mbox{ a.s.}
\end{eqnarray*}
\end{lemma}
\begin{lemma}\label{lem3}{\rm \citep{ibragimov:linnik:1971}}
Let $(\xi_i)_{i\geq 1}$ be a stationary sequence of random variables
with $E\xi_1=0$, $E\lvert\xi_1\rvert^{2+\delta}<\infty$ and $\sum_{k=1}^{\infty}\left( \alpha_k\right) ^{\frac{\delta}{2+\delta}}<\infty$ for some $\delta>0$. Then $\sigma^{2}=E\xi_1^{2}+2\sum_{i=2}^{\infty}E(\xi_{1}\xi_{i})<\infty$ and, in the case $\sigma^{2}>0$,
\begin{equation*}
\dfrac{1}{\sigma\sqrt{n}}\left( \xi_1+...+\xi_n\right)\, \claw \, N(0,1)   \quad \mbox{as } n \to \infty.
\end{equation*}
\end{lemma}
\begin{lemma}{\rm \citep{philipp:1977}}
\label{lem:ep-lil}
Let $(\xi_n)_{n\geq 1}$ be a stationary sequence of strongly mixing random variables with $\alpha_n=O(n^{-8})$, 
and let $\eta_n=f(\xi_n,\xi_{n+1},\ldots)$, $n\geq 1$, be a sequence of functionals with common
distribution function $F(t)$, satisfying
\[
 \psi_m=E\left|\eta_n - E(\eta_n|\F_n^{n+m})  \right| =O(m^{-12}),
\]
where $\F_n^{n+m} = \sigma(\xi_n,\ldots,\xi_{n+m})$. Then
\[
\limsup_{n\rightarrow \infty} \sup_{t\in \R} \frac{\sqrt{n}(\hat{F}_n(t)-F(t))}{\sqrt{2\log\log n}}
\leq c
\]
almost surely, where $\hat{F}_n(t)=\frac{1}{n}\sum_{i=1}^n 1_{\{ \eta_i\leq t  \}}$.
\end{lemma}

%

\begin{thebibliography}{16}
\providecommand{\natexlab}[1]{#1}
\providecommand{\url}[1]{\texttt{#1}}
\expandafter\ifx\csname urlstyle\endcsname\relax
  \providecommand{\doi}[1]{doi: #1}\else
  \providecommand{\doi}{doi: \begingroup \urlstyle{rm}\Url}\fi

\bibitem[Borovkova et~al.(2001)Borovkova, Burton, and
  Dehling]{borovkova:burton:dehling:2001}
S.~Borovkova, R.~M. Burton, and H.~Dehling.
\newblock Limit theorems for functionals of mixing processes with applications
  to {$U$}-statistics and dimension estimation.
\newblock \emph{Transactions of the American Mathematical Society},
  353\penalty0 (11):\penalty0 4261--4318, 2001.

\bibitem[Bradley(2007)]{bradley:2007}
R.~C. Bradley.
\newblock \emph{{Introduction to strong mixing conditions. Vol. 1-3}}.
\newblock {Heber City, UT: Kendrick Press}, 2007.

\bibitem[Carlstein(1986)]{carlstein:1986}
E.~Carlstein.
\newblock The use of subseries values for estimating the variance of a general
  statistic from a stationary sequence.
\newblock \emph{Annals of Statistics}, 14:\penalty0 1171--1179, 1986.

\bibitem[Croux and Dehon(2010)]{croux:dehon:2010}
C.~Croux and C.~Dehon.
\newblock Influence functions of the {Spearman} and {Kendall} correlation
  measures.
\newblock \emph{Statistical Methods and Applications}, 19\penalty0
  (4):\penalty0 497--515, 2010.

\bibitem[de~Jong and Davidson(2000)]{dejong:davidson:2000}
R.~M. de~Jong and J.~Davidson.
\newblock Consistency of kernel estimators of heteroscedastic and
  autocorrelated covariance matrices.
\newblock \emph{Econometrica}, 68\penalty0 (2):\penalty0 407--424, 2000.

\bibitem[Dehling and Fried(2012)]{dehling:fried:2011}
H.~Dehling and R.~Fried.
\newblock Asymptotic distribution of two-sample empirical $U$-quantiles with
  applications to robust tests for shifts in location.
\newblock \emph{Journal of Multivariate Analysis}, 105\penalty0 (1):\penalty0
  124--140, 2012.
\newblock \doi{10.1016/j.jmva.2011.08.014}.

\bibitem[Doukhan and Louhichi(1999)]{doukhan:louhichi:1999}
P.~Doukhan and S.~Louhichi.
\newblock A new weak dependence condition and applications to moment
  inequalities.
\newblock \emph{Stochastic Processes and their Applications}, 84:\penalty0
  313--342, 1999.

\bibitem[Doukhan et~al.(2010)Doukhan, Jakubowicz, and Le{\'o}n]{doukhan:2010}
P.~Doukhan, J.~Jakubowicz, and J.~R. Le{\'o}n.
\newblock Variance estimation with applications.
\newblock In I.~Berkes, R.~Bradley, H.~Dehling, M.~Peligrad, and R.~Tichy,
  editors, \emph{{Dependence in Probability, Analysis and Number Theory.}},
  pages 203--231. {Heber City, UT: Kendrick Press}, 2010.

\bibitem[Ibragimov and Linnik(1971)]{ibragimov:linnik:1971}
I.~A. Ibragimov and Y.~V. Linnik.
\newblock \emph{Independent and stationary sequences of random variables.}
\newblock {Groningen: Wolters-Noordhoff}, 1971.

\bibitem[Pearson(1907)]{pearson:1907}
K.~Pearson.
\newblock \emph{Mathematical contributions to the theory of evolution. XVI. On
  further methods of determining correlation}.
\newblock Drapers' company research memoirs: Biometric series IV. London: Dulau
  \& Co., 1907.

\bibitem[Peligrad and Shao(1995)]{peligrad:shao:1995}
M.~Peligrad and Q.-M. Shao.
\newblock {Estimation of the variance of partial sums for $\rho$-mixing random
  variables.}
\newblock \emph{Journal of Multivariate Analysis}, 52\penalty0 (1):\penalty0
  140--157, 1995.
\newblock \doi{10.1006/jmva.1995.1008}.

\bibitem[Peligrad and Suresh(1995)]{peligrad:suresh:1995}
M.~Peligrad and R.~Suresh.
\newblock {Estimation of variance of partial sums of an associated sequence of
  random variables.}
\newblock \emph{Stochastic Processes and their Applications}, 56\penalty0
  (2):\penalty0 307--319, 1995.
\newblock \doi{10.1016/0304-4149(94)00065-2}.

\bibitem[Philipp(1977)]{philipp:1977}
W.~Philipp.
\newblock {A functional law of the iterated logarithm for empirical
  distribution functions of weakly dependent random variables.}
\newblock \emph{Annals of Probability}, 5:\penalty0 319--350, 1977.
\newblock \doi{10.1214/aop/1176995795}.

\bibitem[{R Development Core Team}(2011)]{r-language}
{R Development Core Team}.
\newblock \emph{R: A Language and Environment for Statistical Computing}.
\newblock R Foundation for Statistical Computing, Vienna, Austria, 2011.
\newblock URL \url{http://www.R-project.org}.
\newblock {ISBN} 3-900051-07-0.

\bibitem[Shao and Yu(1993)]{shao:yu:1993}
Q.-M. Shao and H.~Yu.
\newblock Bootstrapping the sample means for stationary mixing sequences.
\newblock \emph{Stochastic Processes and their Applications}, 48:\penalty0
  175--190, 1993.

\bibitem[Yokoyama(1980)]{yokoyama:1980}
R.~Yokoyama.
\newblock Moment bounds for stationary mixing sequences.
\newblock \emph{Zeitschrift f{\"u}r Wahrscheinlichkeitstheorie und verwandte
  Gebiete}, 52:\penalty0 45--57, 1980.

\end{thebibliography}

\end{document}